\documentclass[12pt,a4paper]{article}

\usepackage[T1]{fontenc}
\usepackage[bitstream-charter]{mathdesign}

\usepackage[latin1]{inputenc}							
\usepackage{amsmath}	
\usepackage{latexsym,xypic}								
\usepackage{geometry}
\usepackage{longtable}
\usepackage{array}
\usepackage{fancyvrb,shortvrb}

\usepackage{xcolor}
\definecolor{bl}{rgb}{0.0,0.2,0.6}

\usepackage{sectsty}
\usepackage[compact]{titlesec}
\allsectionsfont{\color{bl}\scshape\selectfont}

\makeatletter							
\def\printtitle{
    {\color{bl} \centering \huge \sc \textbf{\@title}\par}}		
\makeatother							
\def \la {\l_{\scriptsize{\vdash}}}
\def \lb {\l_{\scriptsize{\dashv}}}
\def \ra {\r_{\scriptsize{\vdash}}}
\def \rb {\r_{\scriptsize{\dashv}}}
\def \sol {\dashv}
\def \sag {\vdash}
\def \k {\mathbb{K}}

\def \r {\vartriangleleft}
\def \l {\vartriangleright}

\title{On Crossed Modules in Modified Categories of Interest
\vspace*{10pt}}

\makeatletter							
\def\printauthor{
    {\color{bl} \centering \large \sc \textbf{\@author}\par}}				
\makeatother							

\author{A.F. ASLAN, S. ÇET\.{I}N and  E. Ö.
USLU 	\vspace*{10pt}}

\usepackage{fancyhdr}
\usepackage{lastpage}	
\usepackage[runin]{abstract}		
\abslabeldelim{\quad}						%
\newtheorem{theorem}{Theorem}
\newtheorem{definition}[theorem]{Definition}
\newtheorem{example}[theorem]{Example}

\newtheorem{corollary}[theorem]{Corollary}

\newtheorem{remark}[theorem]{Remark}
\newtheorem{notation}[theorem]{Notation}

\newtheorem{proposition}[theorem]{Proposition}

\newenvironment{proof}[1][Proof]{\noindent\textbf{#1.} }{\ \rule{0.5em}{0.5em}}

\begin{document}
\printtitle

\vspace{1cm}

\printauthor

\bigskip

\textbf{Address:} Department of Mathematics and Computer Sciences, Eskisehir Osmangazi University, Art and
Science Faculty, Eskisehir, Turkey.

\bigskip

\textbf{e-mail addresses:} afaslan@ogu.edu.tr, selimc@ogu.edu.tr, enveruslu@ogu.edu.tr.

\bigskip
\vspace{1cm}

\textbf{Abstract:} We introduce some algebraic structures such as singularity, commutators and central extension in modified categories of interest. Additionally, we introduce the cat$^{1}$-objects with their connection to crossed modules in these categories which gives rise to unify many notions about (pre)crossed modules in various algebras of categories.

\bigskip


\bigskip


\section{Introduction}
Categories of interest were introduced in order to study properties of
different algebraic categories and different algebras simultaneously. The idea
comes from P. G. Higgins \cite{Hig} and the definition is due to M. Barr and
G. Orzech \cite{Orz}. Many categories of algebraic structures are main
examples of these categories (see \cite{ CDL4, CDL5, Orz,Lo3, DA, LoRo}).The
categories of crossed modules and precrossed modules in the category of
groups, respectively, are equivalent to categories of interest as well, in the
sense of \cite{CDL2,JDL1}. Nevertheless, the cat$^{1}$-Lie (associative,
Leibniz, etc.) algebras are not categories of interest. Consequently, in
\cite{BCDU}, Y. Boyac\i\ et al. introduce and study a new type of category of
interest; namely, a category which satisfies all axioms of a category of
groups with operations stated in \cite{Por} except one, which is replaced by a
new axiom; this category satisfies as well two additional axioms introduced in
\cite{Orz} for categories of interest. They called this category as ''Modified
category of Interest'' which will be denoted MCI from now on. The examples are mainly those categories, which are
equivalent to the categories of crossed modules and precrossed modules in the
categories of Lie algebras, Leibniz algebras, associative and associative
commutative algebras . See \cite{bau,jas,CMDA,CC, Ajas,
	XLR,Ljas,DAT,dede,LL,lue,Por}, for more examples.

Crossed modules were introduced by J.H.C Whitehead in \cite{W1} as a model of
homotopy 3-types and used to classify higher dimensional
cohomology groups in \cite{W2}. Since then, whole property adapted to many
algebras. The notions of crossed modules were defined on various algebras such
as, (associative) commutative algebras, Lie algebras, Leibniz algebras,
Lie-Rinehart algebras in \cite{bau,jas,CMDA,CC, Ajas,
	XLR,Ljas,DAT,dede,LL,lue,Por}. The definition of crossed modules in modified
categories of interest unifies all these definition.  As a different model of
homotopy types, Loday defined cat$^{1}$-groups in \cite{Lswf}. The category of
cat$^{1}$-groups and crossed modules are naturally equivalent and this result
was adapted to many algebras, as well. The notions of cat$^{1}$-algebras were
introduced in \cite{ellis}.

In this work our main purpose is to unify the notions of center, singularity,
commutator and central extensions in various categories of (pre)crossed
modules (See \cite{LCE,BCDU,CC,Nor}). For this, first we introduce the
notions center, singularity and central extensions in modified categories of
interest. Inspiring from the equivalence between the categories of
(pre)cat$^{1}$-groups and (pre)crossed modules, we introduce the notion of
(pre)cat$^{1}$-objects and their connection to crossed modules in modified
categories of interest. Then applying those definitions to (pre)cat$^{1}%
$-objects, we get unification of many notions related to (pre)crossed modules
in different types of categories. Additionally, we show that our definitions
coincide with those given in \cite{everart,gj3,hug1}.

The outline of the paper is as follows: In section $2$, we recall the notion
of MCI and some related structures with basic properties. In section $3$, we
introduce the notion of (pre)cat$^{1}$-object in an arbitrary modified category of
interest $\mathbb{C}$ with its connection to crossed modules in $\mathbb{C}$.
Then we introduce the singularity, commutators and central extensions in MCI.
In section $4$, as an application of section $3$ we get the (pre)crossed
module version of the introduced notions.

\textbf{Acknowledgement : }We would like to thank T. Datuashvili for
valuable comments and suggestions while her visit to Eskisehir Osmangazi
University supported by TÜB\.{I}TAK grand 2221 konuk veya akademik izinli
bilim insan\i\ destekleme program\i .

\bigskip
\section{Preliminaries}

We will recall the notions of MCI and main constructions from \cite{BCDU} which are modified versions of those given in \cite{JDL1, DAT, Orz}.

Let $\mathbb{C}$ be a category of groups with a set of operations $\Omega$ and
with a set of identities $\mathbb{E}$, such that $\mathbb{E}$ includes the
group identities and the following conditions hold. If $\Omega_{i}$ is the set
of $i$-ary operations in $\Omega$, then:

\begin{enumerate}
	\item[(a)] $\Omega=\Omega_{0}\cup\Omega_{1}\cup\Omega_{2}$;
	
	\item[(b)] the group operations (written additively : $0,-,+$) are elements of
	$\Omega_{0}$, $\Omega_{1}$ and $\Omega_{2}$ respectively. Let $\Omega
	_{2}^{\prime}=\Omega_{2}\setminus\{+\}$, $\Omega_{1}^{\prime}=\Omega
	_{1}\setminus\{-\}.$ Assume that if $\ast\in\Omega_{2}$, then $\Omega
	_{2}^{\prime}$ contains $\ast^{\circ}$ defined by $x\ast^{\circ}y=y\ast x$ and
	assume $\Omega_{0}=\{0\}$;
	
	\item[(c)] for each $\ast\in\Omega_{2}^{\prime}$, $\mathbb{E}$ includes the
	identity $x\ast(y+z)=x\ast y+x\ast z$;
	
	\item[(d)] for each $\omega\in\Omega_{1}^{\prime}$ and $\ast\in\Omega
	_{2}^{\prime}$, $\mathbb{E}$ includes the identities $\omega(x+y)=\omega
	(x)+\omega(y)$ and $\omega(x\ast y)=\omega(x)\ast\omega(y)$.
\end{enumerate}

Let $C$ be an object of $\mathbb{C}$ and $x_{1},x_{2},x_{3}\in C$:\newline

\noindent\textbf{Axiom 1.} $x_{1}+(x_{2}\ast x_{3})=(x_{2}\ast x_{3})+x_{1}$,
for each $\ast\in\Omega_{2}^{\prime}$.\newline

\noindent\textbf{Axiom 2.} For each ordered pair $(\ast,\overline{\ast}%
)\in\Omega_{2}^{\prime}\times\Omega_{2}^{\prime}$ there is a word $W$ such
that
\begin{gather*}
	(x_{1}\ast x_{2})\overline{\ast}x_{3}=W(x_{1}(x_{2}x_{3}),x_{1}(x_{3}
	x_{2}),(x_{2}x_{3})x_{1},\\
	(x_{3}x_{2})x_{1},x_{2}(x_{1}x_{3}),x_{2}(x_{3}x_{1}),(x_{1}x_{3})x_{2}
	,(x_{3}x_{1})x_{2}),
\end{gather*}
where each juxtaposition represents an operation in $\Omega_{2}^{\prime}$.
\bigskip

\begin{definition}\cite{BCDU}
	\label{catint2}A category of groups with operations $\mathbb{C}$ satisfying
	conditions $(a)-(d)$, Axiom 1 and Axiom 2, is called a \textit{ modified
		category of interest. }
\end{definition}

The difference of this definition from the
original one of category of interest is the identity $\omega(x)\ast
\omega(y)=\omega(x\ast y),$ which is $\omega(x)\ast y=\omega(x\ast y)$ in the
definition of category of interest.

\begin{example}
	The categories \ $\mathbf{Cat}^{\mathbf{1}}$-$\mathbf{Ass}, \ \mathbf{Cat}
	^{\mathbf{1}}$-$\mathbf{Lie}, \ \mathbf{Cat}^{\mathbf{1}}$-$\mathbf{Leibniz}%
	,\ \mathbf{PreCat}^{\mathbf{1}}$-$\mathbf{Ass}, \ $ $\mathbf{PreCat}%
	^{\mathbf{1}}$-$\mathbf{Lie}$ and $\mathbf{PreCat}^{\mathbf{1}}$%
	-$\mathbf{Leibniz}$ are modified categories of interest, which are not
	categories of interest. Also the category of commutative Von Neumann regular
	rings is isomorphic to the category of commutative rings with a unary
	operation $(\ )^{\ast}$ satisfying two axioms, defined in \cite{jan}, which is
	a modified category of interest.
\end{example}

\begin{notation}
From now on, $\mathbb{C}$ will denote an arbitrary modified category of interest.
\end{notation}
	Let $B\in\mathbb{C}$. A subobject of $B$ is called an ideal if it is the
	kernel of some morphism. Then $A$ is an
	ideal of $B$ if and only if  $A$ is a normal subgroup of $B$ and $a\ast b\in A,$ for all $a\in A$, $b\in B$ and $\ast \in\Omega_{2}^{\prime}$.

For $A,B\in\mathbb{C}$ we say that we have a set of actions of $B$ on
$A$, whenever there is a map $f_{\ast}:A\times B\longrightarrow A$, for each
$\ast\in\Omega_{2}$. A split extension of $B$ by $A,$ induces an
action of $B$ on $A$ corresponding to the operations in $\mathbb{C}$. For a
given\ split extension
\[%
\begin{array}
[c]{r}%
\xymatrix{0\ar[r]&A\ar[r]^-{i}&E\ar[r]^-{p}&B\ar[r]&0},
\end{array}
\]
we have
\[%
\begin{array}
[c]{r}%
b\cdot a=s(b)+a-s(b),
\end{array}
\]%
\[%
\begin{array}
[c]{r}%
b\ast a=s(b)\ast a,
\end{array}
\]
for all $b\in B$, $a\in A$ and $\ast\in\Omega_{2}{}^{\prime}.$ Actions defined
by the previous equations are called derived actions of $B$ on $A$.

Given an action of $B$ on $A,$ a semi-direct product $A\rtimes B$ is a
universal algebra, whose underlying set is $A\times B$ and the operations are
defined by
\[%
\begin{array}
[c]{rcl}%
\omega(a,b) & = & (\omega\left(  a\right)  ,\omega\left(  b\right)  ),\\
(a^{\prime},b^{\prime})+(a,b) & = & (a^{\prime}+b^{\prime}\cdot a,b^{\prime
}+b),\\
(a^{\prime},b^{\prime})\ast(a,b) & = & (a^{\prime}\ast a+a^{\prime}\ast
b+b^{\prime}\ast a,b^{\prime}\ast b),
\end{array}
\]
for all $a,a^{\prime}\in A,$ $b,b^{\prime}\in B.$ See \cite{BCDU}, for details.
\begin{example}
	A dialgebra (or diassociative algebra) over a field $\k$, introduced in \cite{DA} is a $\k$-vector space defined with two $\k$-linear maps:
	\begin{align*}
	\sol \, , \, \sag \colon A \otimes A \to A
	\end{align*}
	such that
	\begin{align*}
	(x \sol y) \sol z & = x \sol (y \sag z),  \\
	(x \sol y) \sol z & = x \sol (y \sol z),   \\
	(x \sag y) \sol z & = x \sag (y \sol z),  \\
	(x \sol y) \sag z & = x \sag (y \sag z), \\
	(x \sag y) \sag z & = x \sag (y \sag z) ,
	\end{align*}
for all $x,y,z \in A$.

	Let $A$ and $B$ be two dialgebras. A dialgebra action of $B$ on $A$  is defined with four bilinear maps:
	\begin{align*}
		\la \, , \, \lb \colon B \times A \to A \\
		\ra \, , \, \rb \colon A \times B \to A
	\end{align*}
satisfying the required $30$ axioms. (For details about these axioms see \cite{CMDA})

	The semi-direct product $A \rtimes B$ is the dialgebra whose underlying set is $A \times B$ with usual scalar multiplication, component wise addition and the binary operations defined by
	$$(a,b)\sol (a',b')=(a \sol a' +b \lb a'+a \rb b', b\sol b' ),$$
	$$(a,b)\sag (a',b')=(a \sag a' +b \la a'+a \ra b', b\sag b' ),$$
	for $a,a'\in A$ and $b,b' \in B$.
\end{example}
\begin{theorem}\cite{BCDU}
	\label{semidirect} An action of $B$ on $A$ is a derived action if and only if
	$A\rtimes B$ is an object of $\mathbb{C}$.
\end{theorem}

\begin{proposition}
	\cite{BCDU} A set of actions of $B$ on $A$ in $\mathbb{C}_{G}$ is a set of
	derived actions if and only if it satisfies the following conditions:
	
	\begin{enumerate}

		\item[\textit{1.}] $0\cdot a=a$,
		
		\item[\textit{2.}] $b\cdot(a_{1}+a_{2})=b\cdot a_{1}+b\cdot a_{2}$,
		
		\item[\textit{3.}] $(b_{1}+b_{2})\cdot a=b_{1}\cdot(b_{2}\cdot a)$,
		
		\item[\textit{4.}] $b\ast(a_{1}+a_{2})=b\ast a_{1}+b\ast a_{2}$,
		
		\item[\textit{5.}] $(b_{1}+b_{2})\ast a=b_{1}\ast a+b_{2}\ast a$,
		
		\item[\textit{6.}] $(b_{1}\ast b_{2})\cdot(a_{1}\ast a_{2})=a_{1}\ast a_{2}$,
		
		\item[\textit{7.}] $(b_{1}\ast b_{2})\cdot(a\ast b)=a\ast b$,
		
		\item[\textit{8.}] $a_{1}\ast(b\cdot a_{2})=a_{1}\ast a_{2}$,
		
		\item[\textit{9.}] $b\ast(b_{1}\cdot a)=b\ast a$,
		
		\item[\textit{10.}] $\omega(b\cdot a)=\omega(b)\cdot\omega(a)$,
		
		\item[\textit{11.}] $\omega(a\ast b)=\omega(a)\ast\omega(b)$, \
		
		\item[\textit{12.}] $x\ast y+z\ast t=z\ast t+x\ast y$,
	\end{enumerate}
for each $\omega\in\Omega_{1}^{\prime}$, $\ast\in\Omega_{2}{}^{\prime}$, $b$,
$b_{1}$, $b_{2}\in B$, $a,a_{1},a_{2}\in A$ and for  $x,y,z,t\in A\cup B$
whenever each side of $12$ has a sense.
\end{proposition}

\begin{definition}
	\cite{BCDU} Let $A\in\mathbb{C}$. The center of $A$ is %
	\[%
	\begin{array}
	[c]{rl}%
	Z(A)= & \{z\in A\mid a+z=z+a,\text{ }a+\omega(z)=\omega(z)+a,\text{ }a\ast
	z=0,\text{ }a\ast\omega\left(  z\right)  =0,\\
	& \text{for all}\ a\in A,\ \omega\in\Omega_{1}\ \text{and}\ \ast\in\Omega
	_{2}{}^{\prime}\}.
	\end{array}
	\]
On the other hand, if $A$ is an ideal of $B$, then the centralizer of $A$ in $B$ is the ideal
\[%
\begin{array}
[c]{rl}%
Z(B,A)= & \{b\in B\mid a+b=b+a,a+\omega(b)=\omega(b)+a,a\ast b=0,\text{ }%
a\ast\omega\left(  b\right)  =0,\text{ }\\
& \text{for\ all}\ \ a\in A,\ \omega\in\Omega_{1}\ \text{and}\ \ast\in
\Omega_{2}{}^{\prime}\}.
\end{array}
\]
\end{definition}
A precrossed module in $\mathbb{C}$ is a triple $(C_{1},C_{0},\partial)$ where $C_{0},C_{1}%
	\in\mathbb{C}$, $C_{0}$ has a derived action on $C_{1}$ and
	$\partial:C_{1}\longrightarrow C_{0}$ is a morphism in $\mathbb{C}$ satisfying
	
	\begin{enumerate}
		\item[\textit{a)}] $\partial(c_{0}\cdot c_{1})=c_{0}+\partial(c_{1})-c_{0}$,
		
		\item[\textit{b)}] $\partial(c_{0}\ast c_{1})=c_{0}\ast\partial(c_{1}),$
	\end{enumerate}
	
	\noindent for all $c_{0}\in C_{0}$, $c_{1}\in C_{1}$, and $\ast\in\Omega_{2}{}^{\prime}$. In addition, if
	
	\begin{enumerate}
		\item[\textit{c)}] $\partial(c_{1})\cdot c_{1}^{\prime}=c_{1} +c_{1}^{\prime
		}-c_{1}$,
		
		\item[\textit{d)}] $\partial(c_{1})\ast c_{1}^{\prime}=c_{1}\ast c_{1}%
		^{\prime}$,
	\end{enumerate}
	
	\noindent for all $c_{1},$ $c_{1}^{\prime}\in C_{1}$, and $\ast\in\Omega_{2}{}^{\prime}%
	$, then the triple $(C_{1} ,C_{0},\partial)$ is called a crossed module in
	$\mathbb{C}$.
	\begin{definition}
		A morphism between two (pre)crossed modules $(C_{1},C_{0},\partial
		)\longrightarrow(C_{1}^{\prime},C_{0}^{\prime},\partial^{\prime})$ is a pair $(\mu_{1},\mu_{0})$ of morphisms
		$\mu_{0}:C_{0}\longrightarrow C_{0}^{\prime}$, $\mu_{1}:C_{1}\longrightarrow
		C_{1}^{\prime}$, such that
		
		\begin{enumerate}
			\item[\textit{a)}] $\mu_{0}\partial=\partial^{\prime}\mu_{1} $,
			
			\item[\textit{b)}] $\mu_{1}(c_{0}\cdot c_{1})=\mu_{0}(c_{0})\cdot\mu_{1}%
			(c_{1})$,
			
			\item[\textit{c)}] $\mu_{1}(c_{0}\ast c_{1})=\mu_{0}(c_{0})\ast\mu_{1}%
			(c_{1}),$
		\end{enumerate}
		for all $c_{0}\in C_{0}$, $c_{1}\in C_{1}$ and $\ast\in\Omega_{2}{}^{\prime}$.
		
	\end{definition}
	Consequently, we have the categories $\mathbf{PXMod}\mathbb{(C)}$ of precrossed modules and $\mathbf{XMod}\mathbb{(C)}$ of crossed modules.

\begin{example}
A crossed module in the category of dialgebras is a homomorphism\\ $\partial:D_{1}\longrightarrow D_{0}$
with an action of $D_{0}$ on $D_{1}$ such that\newline

$%
\begin{array}{ccc}
\mathbf{1)} & \partial (d_{0}\lb d_{1})=d_{0}\sol\partial (d_{1}), &  \\
& \partial (d_{0}\la d_{1})=d_{0}\sag\partial (d_{1}), &  \\
& \partial (d_{1}\rb d_{0})=\partial (d_{1})\sol d_{0}, &  \\
& \partial (d_{1}\ra d_{0})=\partial (d_{1})\sag d_{0}, &
\end{array}%
$\newline
\newline

$%
\begin{array}{cc}
\mathbf{2)} & \partial (d_{1})\lb d_{1}^{\prime }=d_{1}\sol d_{1}^{\prime
}=d_{1}\rb\partial (d_{1}^{\prime }), \\
& \partial (d_{1})\la d_{1}^{\prime }=d_{1}\sag d_{1}^{\prime }=d_{1}\ra%
\partial (d_{1}^{\prime }),%
\end{array}%
$\newline

\noindent for all $d_{1},d_{1}^{\prime }\in D_{1}$, $d_{0}\in D_{0}$. The definition covers the definition
given in \cite{CMDA}.
\end{example}

\begin{example}
Let $\partial :D_{1}\longrightarrow D_{0}$ and $\partial ^{\prime
}:D_{1}^{\prime }\longrightarrow D_{0}^{\prime }$ be crossed modules of
dialgebras. The pair $(\mu _{1},\mu _{0})$ consists of dialgebra
homomorphisms $\mu _{1}:D_{1}\longrightarrow D_{1}^{\prime }$, $\mu
_{0}:D_{0}\longrightarrow D_{0}^{\prime }$ which satisfies $\partial ^{\prime }\mu _{1}=\mu _{0}\partial $ and
\[
\begin{array}{c}
\mu _{1}(d_{0}\la d_{1})=\mu _{0}(d_{0})\la\mu _{1}(d_{1}), \\
\mu _{1}(d_{1}\rb d_{0})=\mu _{1}(d_{1})\rb\mu _{0}(d_{0}), \\
\mu _{1}(d_{0}\lb d_{1})=\mu _{0}(d_{0})\lb\mu _{1}(d_{1}), \\
\mu _{1}(d_{1}\ra d_{0})=\mu _{1}(d_{1})\ra\mu _{0}(d_{0}),%
\end{array}%
\]%
for all $d_{1}\in D_{1}$ and $d_{0}\in D_{0}$ is called a morphism between  $\partial :D_{1}\longrightarrow D_{0}$ and $\partial ^{\prime
}:D_{1}^{\prime }\longrightarrow D_{0}^{\prime }$ .
\end{example}
	
\begin{definition}
	Let  $(C_1,C_0,\mu )$
	be a (pre)crossed module in $\mathbb{C}$.
	A (pre)crossed module $(C'_1,C'_0,\mu' )$ is a (pre)crossed submodule of $(C_1,C_0,\mu )$ if $C'_1$ and $C'_0$ are subobjects of $C_1$, $C_0$, respectively, $\mu' = \mu|_{C'_1}$ and the action of $C'_0$ on $C'_1$ is induced by the action of $C_0$ on $C_1$. Additionally if $C'_0$ and $C'_1$ are ideals of $C_0$ and $C_1$ respectively, $c_0 \ast c'_1 \in C'_1$, $c'_0 \ast c_1 \in C'_1$, $c_0 \cdot c'_1 \in C'_1$, $c'_0 \cdot c_1-c_1 \in C'_1$,  for all $c_1 \in C_1$, $c_0 \in C_0$, $c'_1 \in C'_1$, $c'_0 \in C'_0$ then $(C'_1,C'_0,\mu' )$ is called a crossed ideal of $(C_1,C_0,\mu )$.
\end{definition}

\noindent Equivalently,  $(C'_1,C'_0,\mu' )$ is a crossed ideal of $(C_1,C_0,\mu )$ if and only if $(C'_1,C'_0,\mu' )$ is the kernel of some morphism.

\section{ Some Algebraic Structures in MCI}
In this section, first we will introduce the notion of (pre)cat$^1$-objects in a modified category of interest $\mathbb{C}$ and construct the corresponding category $\mathbf{(Pre)Cat^1}(\mathbb{C})$ of (pre)cat$^1$-objects with natural equivalence with the category $\mathbf{(P)Xmod}(\mathbb{C})$ of (pre) crossed modules in $\mathbb{C}$. Then we will introduce the notions of singularity, commutator and central extensions in $\mathbb{C}$. Also we show that the notion of central extension that we introduced in Definition \ref{central extension}   coincides with the definition of
centrality, in the sense of \cite{JanKel}.
\subsection{(Pre){Cat}$^{1}$- Objects in MCI}

\begin{definition}
A precat$^{1}$-object in $\mathbb{C}$ is a triple $(C, \omega_0, \omega_1)$, where  $C \in \mathbb{C}$ and \\ $ \omega_0, \omega_1:C\longrightarrow C$, are morphisms in $\mathbb{C}$ which satisfy
	
\begin{enumerate}
\item[\textit{1)}] $ \omega_0 \omega_1= \omega_1$, $ \omega_1 \omega_0= \omega_0$.
\end{enumerate}
In addition, if
\begin{enumerate}
\item[\textit{2)}] $x\ast y=0$, $x+y-x-y=0$,
\end{enumerate}
 for all $\ast\in\Omega_{2}{}^{\prime}$
 and $x \in ker \omega_0$, $y \in ker \omega_1$,
then the triple $(C, \omega_0, \omega_1)$ is called a cat$^{1}$-object in $\mathbb{C}$.
\end{definition}
Consider the
category, whose objects are cat$^{1}$-objects and morphisms are $\mathbb{C}%
$-morphisms compatible with the maps $\omega_0$ and $ \omega_1$. We will denote this
category by $\mathbf{Cat}^\mathbf{1}(\mathbb{C})$.\\
Also we have the category $\mathbf{PreCat^1}(\mathbb{C})$ of precat$^1$-objects, in the same manner.
\begin{example}
	Let $\mathbb{C}$ be the category of Leibniz algebras. Then a cat$^{1}$-Leibniz
	algebra is a triple $(L, \omega_0, \omega_1)$ consists of a Leibniz algebra $L$ and Leibniz
	algebra homomorphisms $ \omega_0, \omega_1:L\longrightarrow L$ such that,
	
	\begin{enumerate}
		\item[\textit{1)}] $ \omega_0 \omega_1= \omega_1$, \, $\omega_1 \omega_0= \omega_0$,
		
		\item[\textit{2)}] $[x, y]=0=[y, x]$,
	\end{enumerate}
	 for all $x \in ker\omega_0$, $y \in ker\omega_1$.
\end{example}
\begin{example}
	A cat$^1$-dialgebra is a triple $(D, \omega_0, \omega_1)$ consists of a dialgebra $D$ and homomorphisms $ \omega_0, \omega_1:D\longrightarrow D$ such that,
	
	\begin{enumerate}
		\item[\textit{1)}] $ \omega_0 \omega_1= \omega_1$, \, $\omega_1 \omega_0= \omega_0$,
		
		\item[\textit{2)}]  $x\sol  y=0=y \sol x$,
		 $x \sag y=0=y \sag x $,
	\end{enumerate}
	for all $x \in ker\omega_0$, $y \in ker\omega_1$.
\end{example}

\begin{proposition}
The categories $\mathbf{XMod} \mathbb{(C)} $ and $\mathbf{Cat}^\mathbf{1} \mathbb{(C)} $ are naturally equivalent.
\end{proposition}

\begin{proof}
Let $(C_{1},C_{0},\partial )$ be a crossed module in $\mathbb{C}$. Consider
the corresponding semi-direct product $C_{1}\rtimes C_{0}$ induced from the
action of $C_{0}$ on $C_{1}$. By Theorem \ref{semidirect}, $C_{1}\rtimes
C_{0}\in \mathbb{C}$. It is obvious that the maps $\omega _{0}:C_{1}\rtimes
C_{0}\longrightarrow C_{1}\rtimes C_{0}$, $\omega _{1}:C_{1}\rtimes
C_{0}\longrightarrow C_{1}\rtimes C_{0}$ defined by $\omega
_{0}(c_{1},c_{0})=(0,c_{0})$, $\omega _{1}(c_{1},c_{0})=(0,\partial
(c_{1})+c_{0})$, for all $(c_{1},c_{0})\in C_{1}\times C_{0}$ are $\mathbb{C%
}$-morphisms. On the other hand, since%
\[
\omega _{0}\omega _{1}(c_{1},c_{0})=\omega _{0}(0,\partial
(c_{1})+c_{0})=(0,\partial (c_{1})+c_{0})=\omega _{1}(c_{1},c_{0})
\]%
\newline
and%
\[
\omega _{1}\omega _{0}(c_{1},c_{0})=\omega _{1}(0,c_{0})=(0,c_{0})=\omega
_{0}(c_{1},c_{0}),
\]%
for all $(c_{1},c_{0})\in C_{1}\times C_{0}$, we have $\omega _{0}\omega
_{1}=\omega _{1}$, $\omega _{1}\omega _{0}=\omega _{0}$. Let $%
(c_{1},c_{0})\in ker\omega _{0}$ and $(\overline{c_{1}},\overline{c_{0}})\in
ker\omega _{1}$. Then we have $c_{0}=0$ and $\partial (\overline{c_{1}})+%
\overline{c_{0}}=0$. Consequently,
\begin{align*}
(c_{1},c_{0})+(\overline{c_{1}},\overline{c_{0}})& =(c_{1}+c_{0}.\overline{%
c_{1}},c_{0}+\overline{c_{0}}) \\
& =(c_{1}+\overline{c_{1}},\overline{c_{0}}) \\
& =(\overline{c_{1}}-\overline{c_{1}}+c_{1}+\overline{c_{1}},\overline{c_{0}}%
) \\
& =(\overline{c_{1}}+(-\partial (\overline{c_{1}}))\cdot c_{1},\overline{%
c_{0}}) \\
& =(\overline{c_{1}}+\overline{c_{0}}\cdot c_{1},\overline{c_{0}}+c_{0}) \\
& =(\overline{c_{1}},\overline{c_{0}})+(c_{1},c_{0})
\end{align*}%
and
\begin{align*}
(c_{1},c_{0})\ast (\overline{c_{1}},\overline{c_{0}})& =(c_{1}\ast \overline{%
c_{1}}+c_{1}\ast \overline{c_{0}}+c_{0}\ast \overline{c_{1}},c_{0}\ast
\overline{c_{0}}) \\
& =(c_{1}\ast \overline{c_{1}}+c_{1}\ast \overline{c_{0}}+0\ast \overline{%
c_{1}},0\ast \overline{c_{0}}) \\
& =(c_{1}\ast (\partial (\overline{c_{1}}))+c_{1}\ast \overline{c_{0}},0) \\
& =(c_{1}\ast (\partial (\overline{c_{1}})+\overline{c_{0}}),0) \\
& =(c_{1}\ast 0,0) \\
& =(0,0),
\end{align*}%
as required. So we have the functor $\mathfrak{C}:\mathbf{XMod}\mathbb{(C)}%
\longrightarrow \mathbf{Cat}^{\mathbf{1}}\mathbb{(C)}$.\newline
Conversely, given a cat$^{1}$-object $(C,\omega _{0},\omega _{1})$ in $%
\mathbb{C}$. Consider the morphism $\partial :C_{1}\longrightarrow C_{0}$
where $C_{1}=ker\omega _{0}$, $C_{0}=Im\omega _{0}$ and $\partial =\omega
_{1}\left\vert _{\ker \omega _{0}}\right. $. Define the dot action of $C_{0}$
on $C_{1}$ by $c_{0}\cdot c_{1}=c_{0}+c_{1}-c_{0}$ and the star actions by $%
c_{0}\ast c_{1}$, for $c_{0}\in C_{0}$, $c_{1}\in C_{1}, \, \ast \in \Omega'_2$. We claim that  $(C_{1},C_{0},\partial )$ is a crossed module in $%
\mathbb{C}$ with these actions.

By a direct calculation we have $\omega_0(c_{1})=0$ and there exist  $c\in C$
such that $\omega_0(c)=c_{0}$, for all $c_{0}\in C_{0}$, $c_{1}\in C_{1}$.
\newline
\textbf{i)} For all $c_{0}\in C_{0}$, $c_{1}\in C_{1}$, we have
\begin{align*}
\partial(c_{0}.c_{1}) & = \omega_1(c_{0}+c_{1}-c_{0}) \\
& = \omega_1( \omega_0(c)+c_{1}- \omega_0(c)) \\
& = \omega_1 \omega_0(c)+ \omega_1(c_{1})- \omega_1 \omega_0(c) \\
& = \omega_0(c)+ \omega_1(c_{1})- \omega_0(c) \\
& =c_{0}+\partial(c_{1})-c_{0} .
\end{align*}
\textbf{ii)} For all $c_{0}\in C_{0}$, $c_{1}\in C_{1}$, we have
\begin{align*}
\partial(c_{0}\ast c_{1}) & = \omega_1( \omega_0(c)\ast c_{1}) \\
& = \omega_1 \omega_0(c)\ast \omega_1(c_{1}) \\
& = \omega_0(c)\ast \omega_1(c_{1}) \\
& =c_{0}\ast\partial(c_{1}) .
\end{align*}
\textbf{iii)} Since $\omega_1\omega_1=\omega_1\omega_0\omega_1=\omega_0%
\omega_1=\omega_1$, we have  $\omega_1(c_{1}-\partial(c_{1})) = 0$,  which
means $(c_{1}-\partial(c_{1}))\in ker \omega_1$ and $(c_{1}-\partial
(c_{1}))+c_{1}^{\prime}-(c_{1}-\partial(c_{1}))-c_{1}^{\prime}=0$, for all  $%
c_{1}^{\prime}\in C_{1}$. Then,
\begin{align*}
\partial(c_{1}).c_{1}^{\prime} & =\partial(c_{1})+c^{\prime
}_{1}-\partial(c_{1}) \\
& =c_{1}-c_{1}+\partial(c_{1})+c_{1}^{\prime}-\partial (c_{1}) \\
& =c_{1}+c_{1}^{\prime}-c_{1},
\end{align*}
for all $c_{1},c_{1}^{\prime}\in C_{1}$ as required. \newline
\textbf{iv) }By a calculation similar to (iii) we have  $\partial(c_{1}\ast
c_{1}^{\prime})=\partial(c_{1})\ast c_{1}^{\prime}=c_{1}\ast c_{1}^{\prime}$%
, for all  $c_{1},c_{1}^{\prime}\in C_{1}, \, \ast \in \Omega'_2$.\newline
Consequently, we have the functor $\mathcal{X}:\mathbf{Cat}^\mathbf{1}
\mathbb{(C)} \longrightarrow \mathbf{XMod}\mathbb{(C)} $. The functors $%
\mathfrak{C}$ and $\mathfrak{X}$ give rise to a natural equivalence between $%
\mathbf{XMod}\mathbb{(C)}$ and $\mathbf{Cat}^\mathbf{1} \mathbb{(C)}$
\end{proof}\\
  By a similar way, we have the natural equivalence between $\mathbf{Precat}^\mathbf{1} \mathbb{(C)}$ and $\mathbf{PXMod}\mathbb{(C)}$.
\subsection{ Singularity, Commutators and Central Extensions}
In this section we introduce the notions of singularity, commutators and central extensions in MCI.
\subsubsection{Singularity and Commutators}
\begin{definition} \label{Singular}
	An object $C$ in $\mathbb{C}$ which coincides with its center is called
	singular.
\end{definition}
\begin{example}
Let $A$ be a dialgebra. Then the center $Z(A)$ of $A$ is the set
\begin{equation*}
\{z\in A\left\vert {}\right. a\sol z=0=z\sol a,\,\,a\sag z=0=z\sag a,\,\,%
\text{for all }a\in A\}.
\end{equation*}%
Consequently, $A$ is singular if $a\sol a^{\prime }=0=a\sag a^{\prime
}$, for all $a,a^{\prime }\in A$ .
\end{example}
\begin{example}
	Consider a cat$^1$-group $(G,\omega_0, \omega_1)$. Then $(G,\omega_0, \omega_1)$ is singular if\\ $g+g'=g'+g$, $g+ \omega_i (g')=\omega_i (g')+g$, for all $g,g' \in G$, $i=0,1$.
\end{example}
\begin{definition}
	Let $A \in \mathbb{C}$ and $S \subseteq A$. The smallest ideal containing $S$ will be called the ideal generated by $S$ and denoted by $ < S > $.	
\end{definition}

\begin{definition}
	Let $A\in \mathbb{C}$ and $B,C$ be ideals of $A$ then the ideal generated by
	the set $%
	\begin{array}{l}
	\left\{ b+c-b-c,\,b\ast c,\,b+\omega (c)-b-\omega (c),\,c+\omega
	(b)-c-\omega (b),\,b\ast \omega (c),\,c\ast \omega (b)\left\vert {}\right.
	\right.  \\
	\left. b\in B,c\in C\right\} \text{{will be called the commutator object of }%
		B {and }C. } \\
	\end{array}%
	$
\end{definition}

Let $A\in \mathbb{C}$. The ideal generated by the set $\{x+y-x-y,x+\omega
(y)-x-\omega (y),x\ast y,x\ast \omega (y)\left\vert {}\right. x,y\in A,\ast
\in \Omega _{2}^{\prime }\}$ will be called the commutator of $A$ and denoted by $[A,A]$. Also, $A/[A,A]$ will be called the singularization of $A$.

\begin{example}
Let $D$ be a dialgebra. The commutator of $D$ is is the ideal generated by
the set $\{a\sol b,b\sag a\left\vert {}\right. a,b\in D\}$. Additionally,
the singularization of $D$ is
\begin{equation*}
D/\left\langle a\sol b,b\sag a;a,b\in D\right\rangle .
\end{equation*}
\end{example}

\begin{proposition}
An object $C \in \mathbb{C}$ is singular if and only if $[C,C]=0$.
\end{proposition}

\begin{proof}
Direct checking.
\end{proof}

\begin{remark}	
	The definition of commutators in $\mathbb{C}$ coincides with the Huq's commutator \cite{hug1} and the relative commutator (see \cite{everart}) with the Birkhoff subcategory $\mathbf{Ab}(\mathbb{C})$
	 of singular objects in $\mathbb{C}$.
\end{remark}
\begin{theorem}
	For any object $A \in \mathbb{C}$, the commutator ideal $[A,A]$ is the unique smallest ideal for which makes $A/[A,A]$ singular.
\end{theorem}
\begin{proof}
Direct checking.
\end{proof}
\newline

Denote the full subcategory consists of all singular objects in $\mathbb{C}$
by $\mathbf{Ab}(\mathbb{C})$. We have the functor $\mathfrak{Sing}:\mathbb{C}%
\longrightarrow \mathbf{Ab}(\mathbb{C})$ which takes any object $C$ to its
singularization $C/[C,C]$. Additionally, we have the functor $\mathfrak{inc.}%
:\mathbf{Ab}(\mathbb{C})\longrightarrow \mathbb{C}$ which is the inclusion
of the Birkhoff variety $\mathbf{Ab}(\mathbb{C})$ in $\mathbb{C}$.
Consequently we have the adjunction $%
\begin{array}{c}
\text{\textquotedblleft }\mathfrak{Sing}\dashv \mathfrak{inc.}\text{%
\textquotedblright }%
\end{array}%
$\ which can be diagrammed by%
\begin{equation*}
\xymatrix{ \mathbb{C} \ar@<1ex>[rr]^-{\mathfrak{Sing}} & &
\mathbf{Ab}(\mathbb{C}) \ar@<1ex>[ll]^-{\mathfrak{inc.}} }\text{ }.
\end{equation*}

\subsubsection{Central Extensions}
\begin{definition} \label{central extension}
Let	$C \in \mathbb{C}$ and $A \in  \mathbf{Ab}(\mathbb{C})$. $A$ central extension of $C$ by $A$ is an extension
	\begin{equation*}
		E:\xymatrix{A \ \ar@{>->}[r] & B \ar@{->>}[r] & C}
	\end{equation*}%
	such that $A$ is an subobject of $Z(B)$.
\end{definition}

Janelidze and Kelly \cite{JanKel} introduced the central extension in an exact
category, relative to an "admissible" subcategory. From \cite{jansem}, any modified category of interest $\mathbb{C}$ is Barr exact Mal'tsev category and so any Birkhoff subcategory of $\mathbb{C}$ is admissible which gives rise to consider the categorical theory of
central extensions in $\mathbb{C}$.

An extension $f:A\longrightarrow B$ is called trivial, in the sense of
\cite{JanKel}, if the diagram%
\begin{equation*}
	\xymatrix{A \ar[d]_-f \ar[r] & {{\mathfrak{Sing}}(A)}
		\ar[d]^-{{\mathfrak{Sing}}(f)} \\ B \ar[r] & {{\mathfrak{Sing}}(B)} }
\end{equation*}%
is pullback, where the horizontal morphisms are given by the unit of the
adjunction. An extension is called central, in the sense of \cite{JanKel}, if there exists an extension\\ $\rho
:E\longrightarrow B$ of $B$ such that in the pullback%
\begin{equation*}
	\xymatrix{{E\times_B A} \ar[d]_-{\pi _{1}} \ar[r]^-{\pi _{2}} & A
		\ar[d]^-{f} \\ {E} \ar[r]_-{\rho } & {B} }
\end{equation*}
the morphism $\pi _{1}$ is a trivial extension.

\begin{proposition}\label{xcenter}
	Definition \ref{central extension} coincides with the definition of
	centrality given in \cite{JanKel}. (Here, we consider the category $
	\mathbb{C}$ and the admissible subcategory ${\mathbf{Ab}}(\mathbb{C})$)

\end{proposition}

\begin{proof}
	Let%
	\begin{equation*}
		\xymatrix{A \ \ar@{>->}[r] & B \ar@{->>}[r] & C}
	\end{equation*}%
	be an extension in $
	\mathbb{C}$ with $A\subset Z(B)$. Consider the pullback diagram%
	\begin{equation*}
		\xymatrix{{B\times_C B} \ar[d]_-{\pi _{1}} \ar[r]^-{\pi _{2}} & B
			\ar[d] \\ {B} \ar[r] & {C} }
	\end{equation*}%
	By a direct calculation, the diagram%
	\begin{equation*}
		\xymatrix{{B\times_C B} \ar[d]_-{\pi _{1}} \ar[r] &
			{{\mathfrak{Sing}({{B\times_C B}})}}
			\ar[d]^-{{\mathfrak{Sing}}({\pi _{1}})} \\ {{C}} \ar[r] &
			{\mathfrak{Sing}({C})} }
	\end{equation*}
	is pullback, that is, there exist an isomorphism between $B{\times }_{C}B$
and the fiber product $C{\times }_{\mathfrak{Sing}(C)}\mathfrak{Sing}(B{%
\times }_{C}B)$ defined by $(b,b^{\prime })\longmapsto (b,\overline{%
(b,b^{\prime })})$. So the morphism \\$\pi _{1}:{B\times _{C}B}\longrightarrow
C$ is trivial extension from which we get the centrality in the sense of
\cite{JanKel}.
	
	Conversely, given an extension%
	\begin{equation*}
		\xymatrix{A \ \ar@{>->}[r] & B \ar@{->>}[r]^-{\vartheta _{B}} & C}
	\end{equation*}%
	in $\mathbb{C}$ which is central in the sense of \cite{JanKel}. Then there exists an extension
	$\xymatrix{ E \ar@{->>}[r]^-{\vartheta _{E}} & C}$ such that in the pullback%
	\begin{equation*}
		\xymatrix{{E\times_C B} \ar[d]_-{\pi _{1}} \ar[r]^-{\pi _{2}} & B
			\ar[d] \\ {E} \ar[r] & {C} }
	\end{equation*}%
	the morphism $\pi _{1}:{E\times_C B}\longrightarrow C$ is a
	trivial extension, in other words, the diagram%
	\begin{equation*}
		\xymatrix{{E\times_C B} \ar[d]_-{\pi _{1}} \ar[r]^-{\pi _{2}} &
			{{\mathfrak{Sing}({E\times_C B})}}
			\ar[d]^-{{\mathfrak{Sing}}({\pi _{1}})} \\ {E} \ar[r] &
			{\mathfrak{Sing}({E})} }
	\end{equation*}%
	is pullback. The kernel of $\pi _{1}$ is the injection $\xymatrix{A \ \ar@{>->}[r] &
	{{E \times_C B}}}$ and the kernel of ${\mathfrak{Sing}}(\pi _{1})$
	is the injection $\sigma :A\longrightarrow {\mathfrak{Sing}
		(E\underset{C}{\times} B})$, defined by $\sigma (a)=\overline{(0,a)}$ where $\overline{(0,a)}$ denotes
	the related coset. We want to show that $A\subset Z(B)$. For this we need to show $%
	b+a=a+b,$ $b+\omega (a)=\omega (a)+b,$ $b\ast a=0,$ $b\ast \omega (a)=0$ for
	all $a\in A,b\in B,\omega \in \Omega _{1},\ast \in \Omega _{2}^{\prime }.$
	For all $b\in B$ there exists $e\in E$ such that $\varphi _{B}(b)=\varphi
	_{E}(e).$ Since
		
		\begin{align*}
			\sigma (b+a-b-a) & =\overline{(0,b+a-b-a)} \\
			& =\overline{(0,b)}+\overline{(e,a)}-\overline{(0,b)}-\overline{(e,a)} \\
			& =\overline{(0,b)-(0,b)}+\overline{(e,a)-(e,a)} \\
			& =\overline{(0,0)}
		\end{align*}
	we have $b+a-b-a=0$. By similar calculations we get that $A\subseteq Z(B)$,
	as required.
\end{proof}

 \section{ Applications to (Pre)Crossed Modules in MCI}

  In this section, we introduce the notions of center, singularity and central extension of a (pre)crossed modules in modified categories of interest. For this, we inspired from the equivalence of the categories $\mathbf{(Pre)Cat^1}(\mathbb C)$ of (pre)cat$^1$-objects and $\mathbf{(P)Xmod}(\mathbb C)$ of (pre)crossed modules. In the case of precrossed modules of groups (Lie algebras), the notions give the definitions of centers, singularity and central extensions \cite{LCE,Ljas,CL,PHDN,Nor}.

\subsection{Center and Singularity of Precrossed Modules in MCI }

Let $(C_{1},C_{0},\partial )$ be a precrossed module and $(C_{1}\rtimes C_{0},\omega _{0},\omega _{1})$ be the corresponding precat$^{1}$-object. The center $Z(C_{1}\rtimes C_{0},\omega _{0},\omega _{1})$ of $(C_{1}\rtimes
C_{0},\omega _{0},\omega _{1})$ is the ideal \newline

\noindent $%
\begin{array}{l}
Z(C_{1}\rtimes C_{0},\omega _{0},\omega _{1})=\{(z_{1},z_{0})\in
C_{1}\rtimes C_{0}\left\vert {}\right. z_{1}+z_{0}\cdot
c_{1}=c_{1}+c_{0}\cdot z_{1},\text{ }z_{1}+c_{1}=c_{1}+z_{1}, \\
\text{ \ \ \ \ \ \ \ \ \ \ \ \ \ \ \ \ \ \ \ \ \ \ \ \ \ \ \ \ \ }%
c_{1}=z_{0}\cdot c_{1},\text{ }c_{1}=\partial (z_{1})\cdot c_{1},\text{ }%
c_{0}+\partial (z_{1})=\partial (z_{1})+c_{0}, \\
\text{ \ \ \ \ \ \ \ \ \ \ \ \ \ \ \ \ \ \ \ \ \ \ \ \ \ \ \ \ \ }(c_{1}\ast
z_{0})+(c_{0}\ast z_{1})+(c_{1}\ast z_{0})=0,\text{ }(c_{1}\ast z_{1})=0,%
\text{ }(c_{1}\ast z_{0})=0, \\
\text{ \ \ \ \ \ \ \ \ \ \ \ \ \ \ \ \ \ \ \ \ \ \ \ \ \ \ \ \ \ }(c_{1}\ast
\partial (z_{1})=0,\text{ }\partial (c_{0}\ast z_{1})=0,\text{ for all }%
(c_{1},c_{0})\in C_{1}\rtimes C_{0},\ast \in \Omega _{2}{}^{\prime }\}%
\end{array}%
$\newline
\newline

\noindent The image $\mathcal{X}(Z(C_{1}\rtimes C_{0},\omega _{0},\omega _{1}))$ is the
precrossed ideal $(Z_{1},Z_{0},\partial \left\vert {}\right. )$ of $%
(C_{1},C_{0},\partial )$ where\newline

\noindent $%
\begin{array}{ll}
Z_{1}= & \{z_{1}\in C_{1}\left\vert {}\right.
z_{1}+c_{1}=c_{1}+z_{1},c_{1}\cdot (\partial (z_{1}))=c_{1},\text{ } \\
& c_{0}+\partial (z_{1})=\partial (z_{1})+c_{0},\,z_{1}=c_{0}\cdot z_{1},%
\text{ }c_{1}\ast z_{1}=0, \\
& c_{1}\ast (\partial (z_{1}))=0,\,c_{0}\ast z_{1}=0,\text{ for all }%
c_{1}\in C_{1},c_{0}\in C_{0},\ast \in \Omega _{2}{}^{\prime }\},%
\end{array}%
$\newline

\noindent and\newline

\noindent $%
\begin{array}{l}
Z_{0}=\{z_{0}\in C_{0}\left\vert {}\right. z_{0}\cdot c_{1}=c_{1},\text{ }%
z_{0}+c_{0}=c_{0}+z_{0},\text{ } \\
\text{ \ \ \ \ \ \ \ \ }c_{1}\ast z_{0}=0,\text{ }c_{0}\ast z_{0}=0,\text{
for all }c_{0}\in C_{0},c_{1}\in C_{1},\ast \in \Omega _{2}{}^{\prime }\}.%
\end{array}%
$\newline

\noindent If $(C_{1},C_{0},\partial )$ is a crossed module, then \newline

\noindent $%
\begin{array}{cl}
Z_{1}= & \{z_{1}\in C_{1}\left\vert {}\right.
z_{1}+c_{1}=c_{1}+z_{1},c_{0}+\partial (z_{1})=\partial
(z_{1})+c_{0},c_{0}\cdot z_{1}=z_{1}, \\
& c_{1}\ast z_{1}=0,c_{0}\ast z_{1}=0,\text{ for all }c_{0}\in
C_{0},c_{1}\in C_{1},\ast \in \Omega _{2}{}^{\prime }\},%
\end{array}%
$\newline

\noindent $%
\begin{array}{cll}
Z_{0}= & \{z_{0}\in C_{0}| & z_{0}\cdot
c_{1}=c_{1},z_{0}+c_{0}=c_{0}+z_{0},c_{1}\ast z_{0}=0, \\
&  & c_{0}\ast z_{0}=0,\text{ for all }c_{0}\in C_{0},c_{1}\in C_{1},\ast
\in \Omega _{2}{}^{\prime }\}.%
\end{array}%
$\newline
\begin{definition} \label{Center}
$(Z_{1},Z_{0},\partial )$ will be called the center of $(C_{1},C_{0},\partial )$.
\end{definition}
\noindent We will denote the center of $%
(C_{1},C_{0},\partial )$ by $Z(C_{1},C_{0},\partial )$.

The notions of commuting morphisms and central objects were defined by Huq
\cite{hug1} in the categories with zero objects, products and coproducts,
whose morphisms have images. From these properties following the existence
of injections $\Gamma_i : B_{i} \longrightarrow B_{1} \times B_{2} , i=1,2 $
in the direct product in such a category, we have the following. .

\begin{definition}
	\cite{hug1} Two coterminal morphisms $\beta _{1}:B_{1}\longrightarrow A$ and
	$\beta _{2}:B_{2}\longrightarrow A$ are said to commute if there exists a
	morphism%
	\begin{equation*}
		\beta _{1}\circ \beta _{2}:B_{1}\times B_{2}\longrightarrow A
	\end{equation*}%
	making the diagram
	
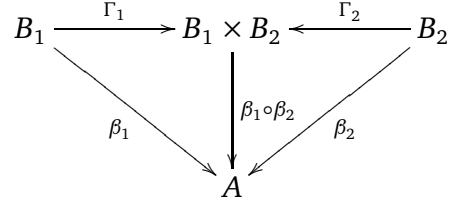
\begin{equation*}
		\xymatrix@R=45pt@C=45pt{
			B_{1} \ar[dr]_-{\beta_{1}} \ar[r]^-{\Gamma_{1}} & B_{1} \times B_{2} \ar[d]^-{\beta_{1} \circ \beta_{2}} & B_{2} \ar[l]_-{\Gamma_{2}} \ar[dl]^-{\beta_{2}} \\
			& A  &  }
	\end{equation*}%
	commutative, where $\Gamma _{i}$, $i=1,2$ denotes the injection of the direct product.
	In particular, a morphism $\beta :B\longrightarrow
	A$ said to be central if the identity morphism on $A$ commutes with $\beta $,
	i.e., if it makes the diagram%
	
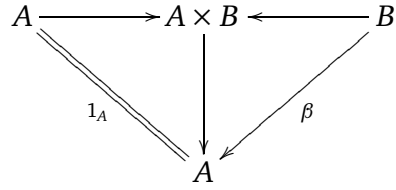
\begin{equation*}
		\xymatrix@R=45pt@C=45pt{
			A \ar@{=}[dr]_-{1_{A}} \ar[r]^-{} & A \times B \ar[d]^-{} & B \ar[l]_-{} \ar[dl]^-{\beta} \\
			& A  &  }
	\end{equation*}%
	commutative. Additionally, if we have a monomorphism $\beta
	:B\longrightarrow A$, then it is said that $B$ is a central subobject of $A$.
\end{definition}

\begin{definition}
	\label{mono}\cite{hug1} The center of an object is the maximal
	central subobject, relative to the order relation that exists on the set of
	monomorphisms.
\end{definition}

\begin{proposition}
	\label{subobject}Let  $(C_1,C_0,\partial )$ be crossed module. Then $
	Z(C_1,C_0,\partial | )$ is the maximal central subobject of $(C_1,C_0,\partial )$.
\end{proposition}

\begin{proof}
	Consider the diagram
	
	
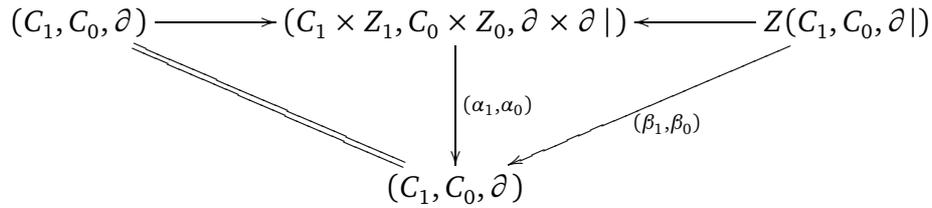
\begin{equation*}
		\xymatrix@R=45pt@C=45pt{
			(C_1,C_0,\partial ) \ar@{=}[dr]_-{} \ar[r]^-{} & (C_1 \times Z_1,C_0 \times Z_0,\partial \times \partial \left\vert {}\right. )  \ar[d]^-{(\alpha_1,\alpha_0) } & Z(C_1,C_0,\partial | ) \ar[l]_-{} \ar[dl]^-{(\beta_1,\beta_0) } \\
			& (C_1,C_0,\partial )  &  }
	\end{equation*}
	
	\noindent  define $\alpha_1 :C_1 \times Z_1 \longrightarrow C_1$, $\alpha_0 :C_0 \times Z_0 \longrightarrow C_0$ by $
	\alpha_1 (c_1 , z_1) = c_1+z_1$, $\alpha_0 (c_0 , z_0) = c_0+z_0$, respectively,   $(\beta_1,\beta_0) $ as an inclusion and the others by usual way. Then the diagram is commutative from which we get that $Z(C_1,C_0,\partial | )$ is a central subobject.
	
		For any central object $(H_1,H_0,\partial | )$ of $(C_1,C_0,\partial )$.
		Then there exist a monomorphism $(\mu _1, \mu_0) :(H_1,H_0,\partial | ) \longrightarrow (C_1,C_0,\partial )$ and a homomorphism $(\sigma _1, \sigma_0) :(C_1 \times H_1,C_0 \times H_0,\partial \times \partial \left\vert {}\right. ) \longrightarrow (C_1,C_0,\partial )$ which makes commutative the diagram
		
    	\begin{equation*}
			\xymatrix@R=45pt@C=45pt{
				(C_1,C_0,\partial ) \ar@{=}[dr]_-{} \ar[r]^-{} & (C_1 \times H_1,C_0 \times H_0,\partial \times \partial \left\vert {}\right. )  \ar[d]^-{(\sigma_1,\sigma_0) } & (H_1,H_0,\partial | ) \ar[l]_-{} \ar[dl]^-{(\mu_1,\mu_0) } \\
				& (C_1,C_0,\partial )  &  }
		\end{equation*}
		By a direct checking we have $(\mu_1,\mu_0)(H_1,H_0,\partial | ) \subseteq Z(C_1,C_0,\partial | ) $, which means that $Z(C_1,C_0,\partial | ) $ is the maximal central subobject of $(C_1,C_0,\partial ) $, as required.
		\end{proof}
	
	\begin{corollary}
		Definition \ref{Center} is equivalent to the definition, in the sense of \cite{hug1}.
	\end{corollary}
	
	\begin{proof}
		Follows from Definitions \ref{mono} and Proposition \ref
		{subobject}.
	\end{proof}

	\begin{definition}
		A singular (pre)crossed module in $\mathbb{C}$ is the crossed module coincide with its center.	
	\end{definition}

    \subsection{The Commutator of a (Pre)Crossed Module in MCI}
In this subsection we introduce the notion of commutator of a precrossed module in $\mathbb{C}$
modules which recovers the Huq's commutator \cite{hug1} and relative commutator \cite{everart}, as well.\newline
Let $(C_{1},C_{0},\partial )$ be a precrossed module. The commutator of the corresponding precat$^{1}$-object $(C_{1}\rtimes C_{0},\omega _{0},\omega _{1})$ is the ideal $[(C_{1}\rtimes C_{0},\omega
_{0},\omega _{1}),(C_{1}\rtimes C_{0},\omega _{0},\omega _{1})]$ generated
by the set \newline
\newline
\noindent $%
\begin{array}{l}
\{(x_{1},x_{0})+(y_{1},y_{0})-(x_{1},x_{0})-(y_{1},y_{0}),\text{ }%
(x_{1},x_{0})+(0,y_{0})-(x_{1},x_{0})-(0,y_{0}), \\
(x_{1},x_{0})+(0,\partial (y_{1})+y_{0})-(x_{1},x_{0})-(0,\partial
(y_{1})+y_{0}),\text{ }(x_{1},x_{0})\ast (y_{1},y_{0}), \\(x_{1},x_{0})\ast
(0,y_{0}), 
(x_{1},x_{0})\ast (0,\partial (y_{1})+y_{0})\left\vert {}\right. \text{ }%
(x_{1},x_{0}),(y_{1},y_{0})\in C_{1}\rtimes C_{0}\text{ and }\ast \in \Omega
_{2}{}^{\prime }\}.%
\end{array}%
$\newline
\newline

The image $\mathfrak{X}([(C_{1}\rtimes C_{0},\omega _{0},\omega _{1}),(C_{1}\rtimes
C_{0},\omega _{0},\omega _{1})])$ is the object $(K_{1},K_{0},\partial |)$
where $K_{1}$ and $K_{0}$ are the ideals generated by the sets
\begin{equation*}
\{x_{0}\cdot x_{1}-x_{1},\text{ }x_{1}+y_{1}-x_{1}-y_{1},\text{ }x_{1}\ast
y_{1},x_{0}\ast x_{1}\left\vert {}\right. x_{0}\in C_{0},\text{ }%
x_{1},y_{1}\in C_{1}\}
\end{equation*}%
and
\begin{equation*}
\{x_{0}+y_{0}-x_{0}-y_{0},\text{ }x_{0}\ast y_{0}\left\vert {}\right.
x_{0},y_{0}\in C_{0}\},
\end{equation*}%
respectively.

\begin{definition}
Let $(C_1,C_0,\partial )$ be a precrossed module. Then $(K_1,K_0,\partial| )$
is called the commutator subcrossed module of $(C_1,C_0,\partial )$.
\end{definition}

If $(C_{1},C_{0},\partial )$ is a crossed module then $K_{1}$ is the set
generated by the set
\begin{equation*}
\{x_{0}\cdot x_{1}-x_{1},x_{0}\ast x_{1}\left\vert {}\right. x_{0}\in C_{0},%
\text{ }x_{1}\in C_{1}\}.
\end{equation*}

	\subsection{Central Extensions of (Pre)Crossed Modules in MCI }
	
	  Now, we introduce the central extensions of (pre)crossed modules in $\mathbb{C}$. Similar to Proposition \ref{xcenter}, the definition coincides with the notion of centrality, in the sense of \cite{JanKel}.
	\begin{definition}
	Let $(C_1,C_0,\partial_C)$ be a (pre)crossed module and $(A_1,A_0,\partial_A)$ be a singular object in $\mathbf{(P)Xmod}(\mathbb{C})$. A central extension of $(C_1,C_0,\partial_C)$ by $(A_1,A_0,\partial_A)$ is an extension
		$$	\xymatrix{(A_1,A_0,\partial_A) \ \ar@{>->}[r] & (B_1,B_0,\partial_B) \ar@{->>}[r] & (C_1,C_0,\partial_C)}  $$
	such that $(A_1,A_0,\partial_A)$ is a crossed ideal of $Z(B_1,B_0,\partial_B)$.
	\end{definition}
As a consequence, one can construct the classification of central extensions of (pre)crossed modules. see \cite{LCE,jas,PHDC,CC,CELE,PHDN,Nor}, for various cases.

\end{document}